\documentclass[12pt,draft]{article}

\usepackage{mathtools}
\usepackage{amssymb}
\usepackage{amsthm}
\usepackage{xcolor}
\usepackage{tikz}
\usepackage{fullpage}
\usepackage{url}

\DeclareSymbolFont{bbold}{U}{bbold}{m}{n}
\DeclareSymbolFontAlphabet{\mathbbold}{bbold}

\newcommand{\N}{\mathbb{N}}

\newcommand{\R}{\mathbb{R}}

\renewcommand{\epsilon}{\varepsilon}

\DeclareMathAccent{\Circ}{\mathalpha}{operators}{"17}
\newcommand{\interior}[1]{\operatorname{\Circ{#1}}}
\newcommand{\inter}[1]{{\Circ{#1}}}
\newcommand{\intersec}[1]{\operatorname{\overbracket[0.5pt]{#1}^{\circ}}}
\newcommand{\ob}[1]{\operatorname{\overbracket[0.5pt]{#1}^{}}}



\DeclareMathOperator{\spt}{spt}

\DeclareMathOperator{\dev}{dev}
\DeclareMathOperator{\sym}{sym}
\DeclareMathOperator{\dive}{div}
\DeclareMathOperator{\curl}{curl}
\DeclareMathOperator{\grad}{grad}
\DeclareMathOperator{\Grad}{Grad}
\DeclareMathOperator{\Dive}{Div}
\DeclareMathOperator{\Curl}{Curl}
\renewcommand{\skew}{\operatorname{skew}}
\DeclareMathOperator{\kar}{ker}
\DeclareMathOperator{\rge}{ran}
\DeclareMathOperator{\dom}{dom}

\DeclareMathOperator{\tr}{tr}
\renewcommand{\Re}{\operatorname{Re}}
\renewcommand{\hat}{\widehat}

\let\phi\varphi

\let\leq\leqslant

\makeatletter
\def\@row#1,{#1\@ifnextchar;{\@gobble}{&\@row}}
\def\@matrix{%
    \expandafter\@row\my@arg,;%
    \@ifnextchar({\\ \get@in@paren{\@matrix}}{\after@matrix}%
    }
\def\matrixtype#1#2#3{%
    \ifmmode\def\after@matrix{\end{#2}\right#3}%
    \else\def\after@matrix{\end{#2}\right#3$}$\fi\iffalse$\fi
    \left#1\begin{#2}\get@in@paren{\@matrix}%
    }
\def\@column#1,{#1\@ifnextchar;{\@gobble}{\\ \@column}}
\newcommand\vect{}
\def\svect(#1){\left(\begin{smallmatrix}\@column#1,;\end{smallmatrix}\right)}
\def\vect{\get@in@paren{\@vect}}
\def\@vect{\left(\begin{matrix}\expandafter\@column\my@arg,;\end{matrix}\right)}
\def\get@in@paren#1({\def\my@arg{}\def\my@rest{}\def\after@get{#1}\get@arg}
\let\e@a\expandafter
\def\get@arg#1){\e@a\kl@test\my@rest#1(;}
\def\kl@test#1(#2;{\e@a\def\e@a\my@arg\e@a{\my@arg#1}%
                   \ifx:#2:\let\my@exec\after@get
                   \else\let\my@exec\get@arg
                        \e@a\def\e@a\my@arg\e@a{\my@arg(}%
                        \def@rest#2;%
                   \fi\my@exec}
\def\def@rest#1(;{\def\my@rest{#1\kl@zu}}
\def\kl@zu{)}

\makeatletter
\newcommand\MyPairedDelimiter{%
  \@ifstar{\My@Paired@Delimiter{{}}}
          {\My@Paired@Delimiter{}}%
}
\newcommand\My@Paired@Delimiter[4]{%
  \newcommand#2{%
    \@ifstar{\start@PD{#1}{\delimitershortfall=-1sp}{#3}{#4}}
            {\start@PD{#1}{}{#3}{#4}}%
  }%
}
\newcommand\start@PD[5]{%
  #1\mathopen{\mathpalette\put@delim@helper{\put@delim{#2}{#3}{.}{#5}}}%
  #5%
  \mathclose{\mathpalette\put@delim@helper{\put@delim{#2}{.}{#4}{#5}}}%
}
\newcommand\put@delim@helper[2]{%
  \hbox{$\m@th\nulldelimiterspace=0pt #2#1$}%
}
\newcommand\put@delim[5]{%
  \setbox\z@\hbox{$\m@th#5{#4}$}%
  \setbox\tw@\null
  \ht\tw@\ht\z@ \dp\tw@\dp\z@
  #1#5%
  \left#2\box\tw@\right#3%
}

\makeatother
\MyPairedDelimiter*{\abs}{\lvert}{\rvert}
\MyPairedDelimiter*{\norm}{\lVert}{\rVert}
\MyPairedDelimiter{\set}{\{}{\}}

\theoremstyle{plain} % default
\newtheorem{theorem}{Theorem}[section]
\newtheorem{corollary}[theorem]{Corollary}
\newtheorem{lemma}[theorem]{Lemma}
\newtheorem{proposition}[theorem]{Proposition}
\theoremstyle{definition}

\newtheorem*{definition}{Definition}
\newtheorem{remark}[theorem]{Remark}

\usepackage{enumitem}

\setenumerate[1]{nolistsep} % Only the level 1
\setenumerate[2]{nolistsep} % Only the level 2

\begin{document}

\title{A Functional Analytic Perspective to the div-curl Lemma}

\author{Marcus Waurick}

\date{}

\maketitle

\begin{abstract}
We present an abstract functional analytic formulation of the celebrated $\dive$-$\curl$ lemma found by F.~Murat and L.~Tartar. The viewpoint in this note relies on sequences for operators in Hilbert spaces. Hence, we draw the functional analytic relation of the div-curl lemma to differential forms and other sequences such as the $\Grad\grad$-sequence discovered recently by D.~Pauly and W.~Zulehner in connection with the biharmonic operator.
\end{abstract}



\medmuskip=4mu plus 2mu minus 3mu
\thickmuskip=5mu plus 3mu minus 1mu
\belowdisplayshortskip=9pt plus 3pt minus 5pt

\section{Introduction}

In the year 1978 a groundbreaking result in the theory of homogenisation has been found by Francois Murat and Luc Tartar, the celebrated $\dive$-$\curl$ lemma (\cite{M78} or \cite{T09}):

\begin{theorem}\label{t:dcl} Let $\Omega\subseteq \mathbb{R}^d$ open, $(u_n)_n,(v_n)_n$ in $L^2(\Omega)^d$ weakly convergent. Assume that
\[
   (\dive u_n)_n= (\sum_{j=1}^d \partial_j u_n)_n,\quad (\curl u_n)_n=\left((\partial_j u_n^{(k)}-\partial_k u_n^{(j)})_{j,k}\right)_n
\]
are relatively compact in $H^{-1}(\Omega)$ and $H^{-1}(\Omega)^{d\times d}$, respectively.

Then $(\langle u_n,v_n\rangle_{\mathbb{C}^d})_n$ converges in $\mathcal{D}'(\Omega)$ and we have
\[
   \lim_{n\to\infty} \langle u_n,v_n\rangle_{\mathbb{C}^d} = \langle \lim_{n\to\infty}u_n,\lim_{n\to\infty}v_n\rangle_{\mathbb{C}^d}.
\] 
\end{theorem}

Ever since people were trying to generalise the latter theorem in several directions. For this we refer to \cite{BCM09}, \cite{L17}, \cite{GM08}, and \cite{LR12} just to name a few. It has been observed that the latter theorem has some relationship to the de Rham cohomology, see \cite{T09}. We shall also refer to \cite{X14}, where the Helmholtz decomposition has been used for the proof of the div-curl lemma for the case of 3 space dimensions. We will meet the abstract counter part of the Helmholtz projection in our abstract approach to the div-curl lemma. In any case, the sequence property of the differential operators involved plays a crucial role in the derivation of the div-curl lemma. Note that, however, there are results that try to weaken this aspect, as well, see \cite{G07}. In this note, in operator theoretic terms, we shall further emphasise the intimate relation of the sequence property of operators from vector analysis and the div-curl lemma. In particular, we will provide a purely functional analytic proof of the $\dive$-$\curl$ lemma. More precisely, we relate the so-called ``global'' form (\cite{S16}) of the div-curl lemma to functional analytic realisations of certain operators from vector analysis, that is, to compact sequences of operators in Hilbert spaces. Moreover, having provided this perspective, we will also obtain new variants of the div-curl lemma, where we apply our abstract findings to the Pauly--Zulehner $\Grad\grad$-sequence, see \cite{PZ17} and \cite{QB15}. With these new results, we have paved the way to obtain homogenisation results for the biharmonic operator with variable coefficients, which, however, will be postponed to future research. 

The next section contains the functional analytic prerequisites and our main result itself -- the operator-theoretic version of the div-curl lemma. The subsequent section is devoted to the proof of the div-curl lemma with the help of the results obtained in Section \ref{s:adcl}.
In the concluding section, we will apply the general result to several examples.

\section{An Abstract $\dive$-$\curl$ Lemma}\label{s:adcl}

We start out with the definition of a (short) sequence of operators acting in Hilbert spaces. Note that in other sources sequences are also called ``complexes''. We use the usual notation of domain, range, and kernel of a linear operator $A$, that is, $\dom(A)$, $\rge(A)$, and $\kar(A)$. Occasionally, we will write $\dom(A)$ to denote the domain of $A$ endowed with the graph norm.

\begin{definition}
  Let $H_j$ be Hilbert spaces, $j\in\{0,1,2\}$. Let $A_0 \colon \dom(A_0)\subseteq H_0\to H_1$, and $A_1 \colon \dom(A_1)\subseteq H_1 \to H_2$ densely defined and closed. The pair $(A_0,A_1)$ is called a \emph{(short) sequence}, if $\rge(A_0)\subseteq \kar(A_1)$. We say that the sequence $(A_0,A_1)$ is \emph{closed}, if both $\rge(A_0)\subseteq H_1$ and $\rge(A_1)\subseteq H_2$ are closed. The sequence $(A_0,A_1)$ is called \emph{compact}, if $\dom(A_1)\cap \dom(A_0^*) \hookrightarrow H_1$ is compact.
\end{definition}

We recall some well-known results for sequences of operators in Hilbert spaces, we refer to \cite{PZ17} and the references therein for the respective proofs.

\begin{theorem}\label{t:tb} Let $(A_0,A_1)$ be a sequence. Then the following statements hold:
\begin{enumerate}[label=(\alph*)]
 \item\label{t1} $(A_1^*,A_0^*)$ is a sequence;
 \item\label{t2} $(A_0,A_1)$ is closed if and only if $(A_1^*,A_0^*)$ is closed.
 \item\label{t3} $(A_0,A_1)$ is compact if and only if $(A_1^*,A_0^*)$ is compact;
 \item\label{t4} if $(A_0,A_1)$ is compact, then $(A_0,A_1)$ is closed. 
 \item\label{t5} $(A_0,A_1)$ is compact if and only if both $\dom(A_0)\cap \kar(A_0)^\bot\hookrightarrow \kar(A_0)^\bot$ and $\dom(A_1^*)\cap \kar(A_1^*)^\bot\hookrightarrow \kar(A_1^*)^\bot$ are compact and $\kar(A_0^*)\cap \kar(A_1)$ is finite-dimensional. 
\end{enumerate}
\end{theorem}

Next, we need to introduce some notation.

\begin{definition}
 Let $H_0,H_1$ be Hilbert spaces, $A\colon \dom(A)\subseteq H_0\to H_1$. Then we define the canonical embeddings
 \begin{enumerate}
  \item $\iota_{\rge(A)} \colon \rge(A) \hookrightarrow H_1$;
  \item $\iota_{\kar(A)} \colon \kar(A) \hookrightarrow H_0$;
  \item $\pi_{\rge(A)}\coloneqq \iota_{\rge(A)}\iota_{\rge(A)}^*$;
  \item $\pi_{\kar(A)}\coloneqq \iota_{\kar(A)}\iota_{\kar(A)}^*$.
 \end{enumerate}
\end{definition}

If a densely defined closed linear operator has closed range, it is possible to continuously invert this operator in an appropriate sense. For convenience of the reader and since the operator to be defined in the next theorem plays an important role in the following, we provide the results with the respective proofs. Note that the results are known, as well, see for instance again \cite{PZ17}. 

\begin{theorem}\label{t:tb2} Let $H_0,H_1$ Hilbert spaces, $A\colon \dom(A)\subseteq H_0\to H_1$ densely defined and closed. Assume that $\rge(A)\subseteq H_1$ is closed. Then the following statements hold:
\begin{enumerate}[label=(\alph*)]
 \item\label{tb2a} $B\coloneqq \iota_{\rge(A)}^*A \iota_{\rge(A^*)}$ is continuously invertible;
 \item\label{tb2b} $B^*=\iota_{\rge(A^*)}^*A^* \iota_{\rge(A)}$;
 \item\label{tb2c} the operator $\hat{A^*}\colon H_1 \to \dom(B)^*, \phi\mapsto (v\mapsto \langle \phi,Av\rangle_{H_1})$ is continuous; and $\hat{B^*}\coloneqq \hat{A^*}|_{\rge(A)}$ is an isomorphism that extends $B^*$.
\end{enumerate}
\end{theorem}
\begin{proof}
  We prove \ref{tb2a}. Note that by the closed range theorem, we have $\rge(A^*)\subseteq H_0$ is closed. Moreover, since $\kar(A)^\bot=\rge(A^*)$, we have that $B$ is injective and since $\iota_{\rge(A)}^*$ projects onto $\rge(A)$, we obtain that $B$ is also onto. Next, as $A$ is closed, we infer that $B$ is closed. Thus, $B$ is continuously invertible by the closed graph theorem.
  
  For the proof of \ref{tb2b}, we observe that $B^*$ is continuously invertible, as well. Moreover, it is easy to see that $B^*=A^*$ on $\dom(A^*)\cap \kar(A^*)^{\bot}$, see also \cite[Lemma 2.4]{TW14}. Thus, the assertion follows.
  
  In order to prove \ref{tb2c}, we note that $\hat{A^*}$ is continuous. Next, it is easy to see that $\hat{B^*}$ extends $B^*$. We show that $\hat{B^*}$ is onto. For this, let $\psi\in \dom(B)^*$. Then there exists $w\in \dom(B)$ such that
  \[
      \langle w,v\rangle_{H_0}+\langle Bw,Bv\rangle_{H_1} = \psi(v)\quad (v\in \dom(B)).
  \]
  Define $\phi\coloneqq (B^{-1})^*w + Bw \in \rge(A)$. Then we compute for all $v\in \dom(B)$
  \begin{align*}
     (\hat{B^*}\phi) (v) & = \langle \phi,Bv\rangle_{H_1} \\
       & = \langle (B^{-1})^*w + Bw, Bv\rangle_{H_1} \\
       & = \langle w,B^{-1}Bv\rangle_{H_0} + \langle Bw,Bv\rangle_{H_1} \\
       & = \psi(v).
  \end{align*}
  Hence, $\hat{B^*}\phi = \psi$. We are left with showing that $\hat{B^*}$ is injective. Let $\hat{B^*}\phi=0$. Then, for all $v\in \dom(B)$ we have
  \[
     0=\langle \phi,Bv\rangle_{H_1}. 
  \]
  Hence, $\phi\in \dom(B^*)$ and $B^*\phi=0$. Thus, $\phi=0$, as $B^*$ is one-to-one. Hence, $\hat{B^*}$ is one-to-one.
\end{proof}

\begin{remark}\label{r:dadb} In the situation of the previous theorem, we remark here a small pecularity in statement \ref{tb2c}: One could also define
\[
   \tilde{A^*} \colon H_1 \to \dom(A)^*, \phi\mapsto (v\mapsto \langle \phi, Av\rangle_{H_1})
\]
to obtain an extension of $A^*$. In the following, we will restrict our attention to the consideration of $\hat{A^*}$. The reason for this is the following fact: \[\dom(A)^*\supseteq\rge(\tilde{A^*})\cong \rge(\hat{A^*})\subseteq \dom(B)^*,\] where the identification is given by
\[
   \hat{A^*}\phi \mapsto (\tilde{A^*}\phi)|_{\dom(B)}\quad (\phi\in H_1).
\]
Indeed, let $\phi\in H_1$. Then
\begin{align*}
   \sup_{\substack{v\in \dom(A),\\ \|v\|_{\dom(A)}\leq1}}|(\tilde{A^*}\phi)(v)| & =\sup_{\substack{v\in \dom(A),\\ \|v\|_{\dom(A)}\leq1}}|\langle \phi,Av\rangle_{H_1}|
   \\ & = \sup_{\substack{v\in \dom(A)\cap \kar(A)^{\bot},\\ \|v\|_{\dom(A)}\leq1}}|\langle \phi,Av\rangle_{H_1}|
   \\ & = \sup_{\substack{v\in \dom(B), \\ \|v\|_{\dom(B)}\leq1}}|\langle \phi,Av\rangle_{H_1}|
   \\ & = \sup_{\substack{v\in \dom(B),\\ \|v\|_{\dom(B)}\leq1}}|(\hat{A^*}\phi)(v)|.
\end{align*}
\end{remark}

The latter remark justifies the formulation in the div-curl lemma, which we state next.

\begin{theorem}\label{t:mr}
  Let $(A_0,A_1)$ be a closed sequence. Let $(u_n)_n, (v_n)_n$ in $H_1$ be weakly convergent. Assume
  \[
     (\hat{A_0^*}u_n)_n, (\hat{A_1} v_n)_n
  \]
 to be relatively compact in $\dom(A_0)^*$ and $\dom(A_1^*)^*$, respectively. Further, assume that $\kar(A_0^*)\cap \kar(A_1)$ is finite dimensional.
 
 Then 
 \[
    \lim_{n\to \infty} \langle u_n,v_n\rangle_{H_1} = \langle \lim_{n\to\infty} u_n, \lim_{n\to\infty} v_n\rangle_{H_1}.
 \]
\end{theorem}

We emphasise that in this abstract version of the $\dive$-$\curl$ lemma \emph{no} compactness condition on the operators $A_0$ and $A_1$ is needed. 

On the other hand, it is possible to formulate a statement of similar type without the usage of (abstract) distribution spaces. For this, however, we have to assume that $(A_0,A_1)$ is a \emph{compact} sequence. The author is indebted to Dirk Pauly for a discussion on this theorem. It is noteworthy that the proof for both Theorem \ref{t:mr} and \ref{t:mrcc} follows a commonly known standard strategy to prove the so-called  `Maxwell compactness property', see \cite{W74,P84,BPS16}. 

\begin{theorem}\label{t:mrcc} Let $(A_0,A_1)$ be a compact sequence. Let $(u_n)_n, (v_n)_n$ be weakly convergent sequences in $\dom(A_0^*)$ and $\dom(A_1)$, respectively. 

 Then 
 \[
    \lim_{n\to \infty} \langle u_n,v_n\rangle_{H_1} = \langle \lim_{n\to\infty} u_n, \lim_{n\to\infty} v_n\rangle_{H_1}.
 \]
\end{theorem}

In order to prove Theorem \ref{t:mr} and \ref{t:mrcc} we formulate a corollary of Theorem \ref{t:tb2} first. 

\begin{corollary}\label{cor:tb3} Let $H_0$, $H_1$ be Hilbert spaces, $A\colon \dom(A)\subseteq H_0\to H_1$ densely defined and closed. Assume that $\rge(A)\subseteq H_1$ is closed. Let $B$ be as in Theorem \ref{t:tb}. For $(\phi_n)_n$ in $H_1$ the following statements are equivalent:
\begin{enumerate}[label=(\roman*)]
 \item\label{tb3i} $(\hat{A^*}\phi_n)_n$ is relatively compact in $\dom(B)^*$;
 \item\label{tb3ii} $(\pi_{\rge(A)}\phi_n)_n$ is relatively compact in $H_1$.
\end{enumerate} 
If $(\phi_n)_n$ weakly converges to $\phi$ in $H_1$, then either of the above conditions imply $\pi_{\rge(A)}\phi_n\to \pi_{\rge(A)}\phi$ in $H_1$.
\end{corollary}
\begin{proof}
  From $\rge(A)=\kar(A^*)^{\bot}$ and $\kar(\hat{A^*})=\kar(A^*)$, we deduce that $\hat{A^*}\phi=\hat{A^*}\pi_{\rge(A)}\phi$ for all $\phi\in H_1$. Next, $\hat{A^*}\pi_{\rge(A)}\phi= \hat{B^*}\iota_{\rge(A)}^* \phi$ for all $\phi\in H_1$. Thus, as $\hat{B^*}$ is an isomorphism by Theorem \ref{t:tb2}, we obtain that  \ref{tb3i} is equivalent to $(\iota_{\rge(A)}^* \phi_n)_n$ being relatively compact in $\rge(A)$. The latter in turn is equivalent to \ref{tb3ii}, since $(\iota_{\rge(A)}^* \phi_n)_n$ being relatively compact is (trivially) equivalent to the same property of $(\iota_{\rge(A)}\iota_{\rge(A)}^* \phi_n)_n=(\pi_{\rge(A)} \phi_n)_n$.
  
  The last assertion follows from the fact that $\pi_{\rge(A)}$ is (weakly) continuous. Indeed, weak convergence of $(\phi_n)_n$ to $\phi$ implies weak convergence of $(\pi_{\rge(A)}\phi_n)_n$ to $\pi_{\rge(A)}\phi$. This together with relative compactness implies $\pi_{\rge(A)}\phi_n\to \pi_{\rge(A)}\phi$ with the help of a subsequence argument. 
\end{proof}

\begin{corollary}\label{cor:tb3b}
 Let $H_0$, $H_1$ be Hilbert spaces, $A\colon \dom(A)\subseteq H_0\to H_1$ densely defined and closed. Assume $\dom(A)\cap \kar(A)^{\bot_{H_0}}\hookrightarrow H_0$ compact. Let $(\phi_n)_n$ weakly converging to $\phi$ in $\dom(A^*)$. Then 
$\lim_{n\to\infty} \pi_{\rge(A)}\phi_n=\pi_{\rge(A)}\phi$ in $H_1$.
\end{corollary}
\begin{proof}
We note that -- by a well-known contradiction argument -- $\dom(A)\cap \kar(A)^{\bot_{H_0}}\hookrightarrow H_0$ compact implies the Poincar\'e type inequality\[
   \exists c>0 \forall \phi\in \dom(A)\cap \kar(A)^\bot: \|\phi\|_{H_0}\leq c \|A\phi\|_{H_1}.
\]
The latter together with the closedness of $A$ implies the closedness of $\rge(A)\subseteq H_0$. Thus, Theorem \ref{t:tb2} is applicable. Let $B$ as in Theorem \ref{t:tb2}.

We observe that the assertion is equivalent to  $\lim_{n\to\infty}\iota_{\rge(A)}^*\phi_n=\iota_{\rge(A)}^*\phi$ in $\rge(A)$.
 We compute with the help Theorem \ref{t:tb2} for $n\in \mathbb{N}$ 
 \begin{align*}
    \iota_{\rge(A)}^*\phi_n & = (B^*)^{-1}B^*\iota_{\rge(A)}^*\phi_n 
    \\  & = (B^*)^{-1}\iota_{\rge(A^*)}^*A^* \iota_{\rge(A)}\iota_{\rge(A)}^*\phi_n
    \\ & = (B^*)^{-1}\iota_{\rge(A^*)}^*A^* \pi_{\rge(A)}\phi_n
    \\ & = (B^*)^{-1}\iota_{\rge(A^*)}^*A^*\phi_n.
 \end{align*}
 By hypothesis, $A^*\phi_n\rightharpoonup A^*\phi$ in $H_0$ and so $\iota_{\rge(A^*)}^*A^*\phi_n\rightharpoonup \iota_{\rge(A^*)}^*A^*\phi$ in $\rge(A^*)$ as $n\to\infty$ since $\iota_{\rge(A^*)}^*$ is (weakly) continuous.  Next $B^{-1}$ is compact by assumption and thus so is $(B^*)^{-1}$. Therefore $(B^*)^{-1}\iota_{\rge(A^*)}^*A^*\phi_n\to (B^*)^{-1}\iota_{\rge(A^*)}^*A^*\phi$ in $\iota_{\rge(A)}$. The assertion follows from $(B^*)^{-1}\iota_{\rge(A^*)}^*A^*\phi=\iota_{\rge(A)}^*\phi$.
\end{proof}

\begin{proof}[Proof of Theorem \ref{t:mr} and Theorem \ref{t:mrcc}] By the sequence property, we deduce that $\pi_{\rge(A_0)}\leq \pi_{\kar(A_1)}$ and $\pi_{\rge(A_1^*)}\leq \pi_{\kar(A_0^*)}$. By Corollary \ref{cor:tb3} (Theorem \ref{t:mr}) or Corollary \ref{cor:tb3b} (Theorem \ref{t:mrcc}), we deduce that $\pi_{\rge(A_0)}u_n\to\pi_{\rge(A_0)}u$ and $\pi_{\rge(A_1^*)} v_n\to \pi_{\rge(A_1^*)}v$ in $H_1$. From $\kar(A_1)\cap \kar(A_0^*)$ being finite-dimensional (cf. Theorem \ref{t:tb}), we obtain $\pi_{\kar(A_1)\cap \kar(A_0^*)}u_n\to \pi_{\kar(A_1)\cap \kar(A_0^*)}u$ as $\pi_{\kar(A_1)\cap \kar(A_0^*)}$ is compact. Thus, we obtain for $n\in \mathbb{N}$
\begin{align*}
  \langle u_n, v_n\rangle_{H_1}& = \langle (\pi_{\rge(A_0)}+\pi_{\kar(A_0^*)\cap \kar(A_1)}+\pi_{\kar(A_0^*)\cap \rge(A_1^*)}) u_n, (\pi_{\rge(A_1^*)}+\pi_{\kar(A_1)}) v_n\rangle_{H_1}
  \\   & = \langle u_n, \pi_{\rge(A_1^*)}v_n\rangle_{H_1} 
  \\ & \quad
   +\langle (\pi_{\rge(A_0)}+\pi_{\kar(A_0^*)\cap \kar(A_1)}+\pi_{\kar(A_0^*)\cap \rge(A_1^*)}) u_n,\pi_{\kar(A_1)} v_n\rangle_{H_1}
   \\ & = \langle u_n, \pi_{\rge(A_1^*)}v_n\rangle_{H_1} 
   \\ & \quad+\langle \pi_{\rge(A_0)}u_n,\pi_{\kar(A_1)} v_n\rangle_{H_1} +
   \langle \pi_{\kar(A_0^*)\cap \kar(A_1)} u_n,\pi_{\kar(A_1)} v_n\rangle_{H_1} 
    \\ & \to  \langle \lim_{n\to\infty} u_n,\lim_{n\to\infty} v_n\rangle_{H_1}.\qedhere
\end{align*} 
\end{proof}

A closer look at the proof of our main result reveals the following converse of Theorem \ref{t:mr}:
\begin{theorem}\label{t:mrc}
  Let $(A_0,A_1)$ be a closed sequence. Assume that for all weakly convergent sequences $(u_n)_n, (v_n)_n$ in $\dom(A_0^*)$ and $\dom(A_1)$, respectively, we obtain
 \[
    \lim_{n\to \infty} \langle u_n,v_n\rangle_{H_1} = \langle \lim_{n\to\infty} u_n, \lim_{n\to\infty} v_n\rangle_{H_1}.
 \]
 Then $\kar(A_0^*)\cap \kar(A_1)$ is finite-dimensional.
\end{theorem}

For the proof of the latter, we need the next proposition:
\begin{proposition}\label{p:id} Let $H$ be a Hilbert space. Then the following statements are equivalent:
\begin{enumerate}
 \item $H$ is infinite-dimensional;
 \item there exists $(u_n)_n$ weakly convergent to $0$ such that $c\coloneqq \lim_{n\to\infty}\langle u_n,u_n\rangle$ exists with $c\ne 0$.
\end{enumerate} 
\end{proposition}
\begin{proof}
 Let $H$ be infinite-dimensional. Without loss of generality, we may assume that $H=L^2(0,2\pi)$. Then $u_n\coloneqq \sin(n\cdot)\to 0$ weakly as $n\to\infty$ and
 \[
    \int_0^{2\pi} (\sin(nx))^2dx\to \frac{1}{2\pi}\int_0^{2\pi} (\sin(x))^2dx>0. 
 \]
 If $H$ is finite-dimensional, then weak convergence and strong convergence coincide, and the desired sequence cannot exist.
\end{proof}

\begin{proof}[Proof of Theorem \ref{t:mrc}]
 Suppose that $\kar(A_0^*)\cap \kar(A_1)$ is infinite-dimensional. Choose $(u_n)_n$ in $\kar(A_0^*)\cap \kar(A_1)$ as in Proposition \ref{p:id}. Then, clearly, $(u_n)_n$ is weakly convergent in $\dom(A_0^*)$ and $\dom(A_1)$. Hence,
 \[
 0= \langle \lim_{n\to\infty} u_n, \lim_{n\to\infty} u_n\rangle_{H_1} = \lim_{n\to\infty}\langle  u_n,  u_n\rangle_{H_1} = c \neq 0.\qedhere
 \]
\end{proof}

We will need the next abstract results for the proof of the div-curl lemma in the next section. Note that this is only needed for the formulation of the div-curl lemma where the divergence and the curl operators are considered to map into $H^{-1}$.
For this, we need some notation. Let $A\in L(H_0,H_1)$. The dual operator $A'\in L(H_1^*,H_0^*)$ is given by
\[
   (A'\phi)(\psi)\coloneqq \phi(A\psi). 
\]
We also define $A^\diamond \colon H_1 \to H_0^*$ via $A^\diamond \coloneqq A'R_{H_1}$, where $R_{H_1}\colon H_1\to H_1^*$ denotes the Riesz isomorphism.

\begin{proposition}\label{p:dual} Let $H_0$, $H_1$, $D$ Hilbert spaces, $A\colon \dom(A)\subseteq H_0\to H_1$ densely defined and closed. Assume $D\hookrightarrow\dom(A)$ continuously and $\rge(A|_D)= \rge(A)\subseteq H_1$ closed. Define $\mathcal{A}\colon D\to H_1, \phi\mapsto A\phi$. Then $\hat{A^*}=\mathcal{A}^\diamond$, that is, for every $v\in H_1$ we have $\mathcal{A}^\diamond v$ can be uniquely extended to an element of $\dom(A)^*$, the extension is given by $\hat{A^*}v$, where $\hat{A^*}$ is given in Theorem \ref{t:tb2}.
\end{proposition}
\begin{proof}
 Let $v \in H_1$. Then for all $\phi\in D$ we have
 \[
   \left(\hat{A^*}v\right)(\phi)= \langle v, A\phi\rangle_{H_1}=\langle v, \mathcal{A}\phi\rangle_{H_1} = R_{H_1}v (\mathcal{A}\phi) = (\mathcal{A}'R_{H_1}v)(\phi) = (\mathcal{A}^\diamond v)(\phi).
 \] Since $\mathcal{A}$ is continuous, it is densely defined and closed, hence $\mathcal{B}\coloneqq \iota_{\rge(\mathcal{A})}^*\mathcal{A}\iota_{\rge(\mathcal{A}^*)}$ is a Hilbert space isomorphism from $D\cap \kar(\mathcal{A})^{\bot_D}$ to $\rge(\mathcal{A})=\rge(A)$, by Theorem \ref{t:tb2}. Note that $\mathcal{A}\mathcal{B}^{-1}=\mathrm{id}_{\rge(\mathcal{A})}=\mathrm{id}_{\rge({A})}$. For $\psi\in \dom(A)$ and $v\in H_1$, we define
 \[
     \left(\mathcal{A}^\diamond v\right)_{\textnormal{e}}(\psi)\coloneqq \left(\mathcal{A}^\diamond v\right)(\mathcal{B}^{-1}A\psi).
 \]
 Next, if $\psi\in \dom(A)$, then with the above computations, we obtain
 \[
  \left(\mathcal{A}^\diamond v\right)_{\textnormal{e}}(\psi)= \left(\mathcal{A}^\diamond v\right)(\mathcal{B}^{-1}A\psi)=
  \langle v, \mathcal{A}\mathcal{B}^{-1}A\psi\rangle_{H_1}=\langle v, A\psi\rangle_{H_1}=\left(\hat{A^*}v\right)(\psi).
 \]
 Thus, $\left(\mathcal{A}^\diamond v\right)_{\textnormal{e}}$ indeed extends $\mathcal{A}^\diamond v$ and coincides with $\hat{A^*}v$. We infer also the continuity property for $\mathcal{A}^\diamond v$. The uniqueness property follows from $\rge(\mathcal{A})=\rge(A)$.
%  
%  
%  Next, let $\psi\in\dom(A)$. Then there is $\phi\in D$ with $A\phi=A\psi$. Thus, by the computations above we deduce
%  \[
%     \left(\hat{A^*}v\right)(\psi)=\langle v, A\psi\rangle_{H_1}=\langle v, A\phi\rangle_{H_1}=\left(\hat{A^*}v\right)(\phi)=(\mathcal{A}^\diamond v)(\phi).\qedhere
%  \]
\end{proof}
From Proposition \ref{p:dual} it follows that $\rge(\hat{A^*})=\rge(\mathcal{A}^\diamond)$. This is the actual fact used in the following.

\begin{lemma}[{{\cite[Lemma 2.14]{PZ17}}}]\label{l:PZr} Let $H_0$, $H_1$, $H_2$ Hilbert spaces, $A\in L(H_1,H_2)$ onto. Then 
%\begin{enumerate}
 %\item 
 $\rge(A^\diamond)\subseteq H_1^*$ is closed and $(A^\diamond)^{-1}\in L(\rge(A^\diamond),H_2)$.
 %\item If, in addition, $\kar(A)=\rge(B)$ for some $B\in L(H_0,H_1)$, then $\rge(A^\diamond)=\kar(B^\diamond)$.
%\end{enumerate} 
\end{lemma}
\begin{proof}
 By the Riesz representation theorem $A^\diamond$ and $A'$ are unitarily equivalent. Thus, it suffices to prove the assertions for $A'$ instead of $A^\diamond$.
 By the closed range theorem, $\rge(A')$ is closed, since $\rge(A)=H_2$ is. Next, $A$ is onto, hence $A'\in L(H_2^*,H_1^*)$ is one-to-one, and, thus, by the closed graph theorem, we obtain that $(A')^{-1}$ maps continuously from $\rge(A')$ into $H_2^*$.
\end{proof}

\begin{corollary}\label{cor:AB} Let $H_0$, $H_1$ be Hilbert spaces, $A\colon \dom(A)\subseteq H_0\to H_1$ densely defined and closed, $C\colon \dom(C)\subseteq H_0\to H_1$ densely defined, closed. Assume that $\rge(A)\subseteq H_1$ is closed, $\dom(C)\hookrightarrow \dom(A)$ continuous.

If 
\begin{equation}\label{ABeq}
   \rge(A)=\{A\phi; \phi\in \dom(C)\},
\end{equation}
then $\rge(\hat{A^*})=\dom(B)^*\subseteq \dom(C)^*$ is closed, where $B$ is given in Theorem \ref{t:tb2}.
\end{corollary}
\begin{proof}
  Since $\dom(C)\hookrightarrow \dom(A)$ continuously, we obtain that 
  \[
     {\mathcal{A}} \colon \dom(C)\to \rge(A)=\rge(B), \phi\mapsto A\phi
  \]
 is continuous. Moreover, by \eqref{ABeq}, we infer that ${\mathcal{A}}$ is onto. Hence, by Lemma \ref{l:PZr}, we obtain that $\rge({ {\mathcal{A}}}^\diamond)\subseteq \dom(C)^*$ is closed. Thus, we are left with showing that $\rge({\mathcal{A}}^\diamond)=\dom(B)^*$. By Proposition \ref{p:dual}, we realise that $\rge({\mathcal{A}}^\diamond)=\rge(\hat{A^*})=\rge(\hat{B^*})$. By Theorem \ref{t:tb2}, we get that $\hat{B^*}$ maps onto $\dom(B)^*$.
\end{proof}

\begin{remark}
 Corollary \ref{cor:AB} particularly applies to $A=C$.
\end{remark}

\section{The classical $\dive$-$\curl$ lemma}

Before we formulate Theorem \ref{t:dcl2}, the classical $\dive$-$\curl$ lemma, we need to introduce some differential operators from vector calculus.

\begin{definition} Let $\Omega\subseteq \mathbb{R}^d$ open. We define 
\begin{align*}
   \grad_\text{c} & \colon C_c^\infty(\Omega) \subseteq L^2(\Omega) \to L^2(\Omega)^d, \phi\mapsto (\partial_j\phi_j)_{j\in\{1,\ldots,d\}}
   \\ \dive_\text{c} & \colon C_c^\infty(\Omega) \subseteq L^2(\Omega)^d \to L^2(\Omega), (\phi_j)_{j\in\{1,\ldots,d\}}\mapsto \sum_{j=1}^d\partial_j\phi_j
   \\ \Grad_{\text{c}} & \colon C_c^\infty(\Omega)^d \subseteq L^2(\Omega)^d \to L^2(\Omega)^{d\times d}, (\phi_j)_{j\in\{1,\ldots,d\}}\mapsto (\partial_k\phi_j)_{j,k\in\{1,\ldots,d\}}
   \\ \Dive_{\text{c}} & \colon C_c^\infty(\Omega)^{d\times d} \subseteq L^2(\Omega)^{d\times d} \to L^2(\Omega)^{d}, (\phi_{j,k})_{j,k\in\{1,\ldots,d\}}\mapsto (\sum_{k=1}^d\partial_k\phi_{j,k})_{j\in\{1,\ldots,d\}}
   \\ \Curl_{\text{c}} & \colon C_c^\infty(\Omega)^{d} \subseteq L^2(\Omega)^{d} \to L^2(\Omega)^{d\times d}, (\phi_{j})_{j\in\{1,\ldots,d\}}\mapsto (\partial_k\phi_{j}-\partial_j\phi_{k})_{j,k\in\{1,\ldots,d\}}
   \\& \hspace{10cm}=\Grad \phi-\left(\Grad\phi\right)^T.
\end{align*}
Moreover, we set $\interior{\grad}\coloneqq \overline{\grad}_\text{c}$ and, similarly, $\interior{\dive},\interior{\Dive},\interior{\Curl},\interior{\Grad}$. Furthermore, we put $\dive\coloneqq -\interior{\grad}^*$, $\Dive\coloneqq -\interior{\Grad}^*$, $\grad\coloneqq -\interior{\dive}^*$, $\Grad\coloneqq -\interior{\Dive}^*$ and $\Curl\coloneqq (2\interior{\Dive} \skew)^*$, where $\skew A \coloneqq \frac{1}{2}(A-A^T)$ denotes the skew symmetric part of a matrix $A$.
\end{definition}

\begin{remark}
 It is an elementary computation to establish that the operators just introduced with $\interior{\ }$ are restrictions of the ones without.
\end{remark}

As usual, we define, $H^{-1}(\Omega)\coloneqq \dom(\interior{\grad})^*$. We may now formulate the classical $\dive$-$\curl$ lemma. We slightly rephrase the lemma, though.

\begin{theorem}[$\dive$-$\curl$ lemma -- global version]\label{t:dcl2} Let $(u_n)_n,(v_n)_n$ in $L^2(B(0,1))^d$ weakly convergent, with \[\overline{\bigcup_{n\in\N}(\spt u_n \cup \spt v_n)}\subseteq B(0,\delta)=\{ x\in \R^d; \|x\| \leq\delta\}\] for some $\delta<1$. Assume 
\[
   (\dive u_n)_n, (\Curl u_n)_n 
\] are relatively compact in $H^{-1}(B(0,1))$ and $H^{-1}(B(0,1))^{d\times d}$, resp. 

Then 
\[
   \lim_{n\to\infty}\langle u_n, v_n \rangle_{L^2} = \langle \lim_{n\to\infty}u_n,\lim_{n\to\infty}v_n\rangle_{L^2}.
\] 
\end{theorem}

We recall here that in \cite{S16}, Theorem \ref{t:dcl2} is called ``global $\dive$-$\curl$ lemma''. We provide the connection to the classical, the ``local'' version of it, in the following remark.

\begin{remark}[$\dive$-$\curl$ lemma -- local version]\label{r:dcl12}
  We observe that the assertions in Theorem \ref{t:dcl} and in Theorem \ref{t:dcl2} are equivalent. For this, observe that Theorem \ref{t:dcl} implies Theorem \ref{t:dcl2}. Indeed, for $\Omega=B(0,1)$, the assumptions of Theorem \ref{t:dcl2} imply the same of Theorem \ref{t:dcl}. Moreover, let $\phi\in C_c^\infty(B(0,1))$ be such that $\phi=1$ on the compact set $\overline{\bigcup_{n\in\N}(\spt u_n \cup \spt v_n)}$. Then, by Theorem \ref{t:dcl} and putting $u\coloneqq \lim_{n\to\infty}u_n$ and $v\coloneqq \lim_{n\to\infty}v_n$, we obtain
  \[
    \langle u_n, v_n \rangle_{L^2}=  \int_\Omega \phi \langle u_n,v_n\rangle \to  \int_\Omega \phi \langle u,v\rangle = \langle u,v\rangle.
  \]
  On the other hand, let the assumptions of Theorem \ref{t:dcl} be satisfied. With the help of Theorem \ref{t:dcl2}, we have to prove that for all $\phi\in C_c^\infty(\Omega)$ we get
  \begin{equation}\label{e:dcl12}
   \int_\Omega \phi \langle u_n,v_n\rangle \to  \int_\Omega \phi \langle u,v\rangle.
  \end{equation}
  To do so, we let $\psi\in C_c^\infty(\Omega)$ be such that $\psi=1$ on $\spt \phi$. Then there exists $R>0$ such that $\spt \psi \subseteq B(0,R)$. By  rescaling the arguments, the statement in \eqref{e:dcl12} follows from Theorem \ref{t:dcl2}, once we proved that
  \[
     (\dive(\psi u_n))_n=(\psi\dive(u_n)+\grad(\psi )u_n)_n,\,   (\Curl(\psi v_n))_n=(2\skew((\grad \psi) v_n^T)+\psi \Curl v_n)_n
  \]
  is relatively compact in $H^{-1}(B(0,R+1))$ and $H^{-1}(B(0,R+1))^{d\times d}$. This, however, follows from the hypothesis and the compactness of the embedding $L^2(B(0,1)) \hookrightarrow H^{-1}(B(0,1))$, which in turn follows from Rellich's selection theorem.
\end{remark}

The rest of this section is devoted to prove Theorem \ref{t:dcl2} by means of Theorem \ref{t:mr}. We will apply Theorem \ref{t:mr} to the following setting
\begin{alignat}{1}\tag{$*$}\label{eq:setting}
\begin{aligned}
H_0 & = L^2(B(0,1)),
\\ H_1 & = L^2(B(0,1))^d,
\\ A_0 &\coloneqq \interior{\grad},
\\ A_1 &\coloneqq \interior{\Curl}.
\end{aligned}
\end{alignat}

\begin{proposition}\label{prop:seq} With the setting in \eqref{eq:setting}, $(A_0,A_1)$ is a sequence.
\end{proposition}
\begin{proof}
 By Schwarz's lemma, it follows for all $\phi\in C_c^\infty(B(0,1))$ that
\[
   \interior{\Curl}\interior{\grad}\phi=\interior{\Curl}(\partial_j\phi)_{j\in\{1,\ldots,d\}}=(\partial_k\partial_j\phi-\partial_j\partial_k\phi)_{j,k\in\{1,\ldots,d\}}=0.
\]
Thus, $\interior{\Curl}\interior{\grad}\subseteq 0$. 
\end{proof}
Next, we address the compactness property. 
\begin{theorem}\label{t:comp}
 With the setting in \eqref{eq:setting}, $(A_0,A_1)$ is compact.
\end{theorem}
For the proof of Theorem \ref{t:comp}, we could use compactness embedding theorems such as Weck's selection theorem (\cite{W74}) or Picard's selection theorem (\cite{P84}). However, due to the simple geometric setting discussed here, it suffices to walk along the classical path of showing compactness by proving Gaffney's inequality and then using Rellich's selection theorem. We emphasise, however, that meanwhile there have been developed sophisticated tools detouring Gaffney's inequality, to obtain compactness results for very irregular $\Omega$, which do not satisfy Gaffney's inequality.  For convenience of the reader, we shall provide a proof of Theorem \ref{t:comp} using the following regularity result for the Laplace operator, see \cite[Teorema 10 and 14]{K64} or since we use the respective result for a $d$-dimensional ball, only, see \cite[Inequality (3,1,1,2)]{G11}. For this, we denote the Dirichlet Laplace operator by $\Delta\coloneqq \dive\interior{\grad}$.
\begin{theorem}\label{t:LaplReg}Let $\Omega\subseteq \mathbb{R}^d$ open, bounded and convex. Then for all $u\in \dom(\Delta)$, we have $u\in \dom(\Grad\interior{\grad})$ and
\[
    \|\Grad\interior{\grad}u\|_{L^2(\Omega)^{d\times d}}\leq \|\Delta u\|_{L^2(\Omega)}.
\] 
\end{theorem}

Based on the latter estimate, we shall prove Friedrich's inequality. For the proof of which, we will follow the exposition of \cite{S82}. Since the exposition in \cite{S82} is restricted to 2 or 3 spatial dimensions, only, we provide a proof for the ``multi-$d$''-case in the following.
\begin{theorem}[{{\cite[Theorem 2.2]{S82}}}]\label{t:gaff} Let $\Omega\subseteq \mathbb{R}^d$ open, bounded, convex. Then $\dom(\interior{\Curl})\cap\dom(\dive)\hookrightarrow \dom(\Grad)$. Moreover, we have 
\[
   \|\Grad u\|_{L^2(\Omega)^d}^2\leq \frac12\|\interior{\Curl} u\|^2_{L^2(\Omega)^{d\times d}}+\|\dive u\|^2_{L^2(\Omega)} 
\] 
for all $u\in \dom(\interior{\Curl})\cap\dom(\dive)$.
\end{theorem}

\begin{lemma}[{{\cite[Lemma 2.1]{S82}}}]\label{l:dens} Let $\Omega\subseteq \mathbb{R}^d$ open, bounded. Denote
\[
   V\coloneqq \{ \phi; \exists \psi\in C_c^\infty(\Omega)^d:\, \phi=\psi+\interior{\grad}(-\Delta+1)^{-1}\dive \psi \}.
\]
Then $V$ is dense in $\dom(\interior{\Curl})\cap \dom(\dive)$.
\end{lemma}
\begin{proof}
 First of all note that $V\subseteq X\coloneqq \dom(\interior{\Curl})\cap \dom(\dive)$. Indeed, for $\phi=\psi+\interior{\grad}(-\Delta+1)^{-1}\dive \psi$ for some $\psi\in C_c^\infty(\Omega)$, we get $\interior{\Curl}\phi=\interior{\Curl} \psi\in L^2(\Omega)^{d\times d}$, by Proposition \ref{prop:seq}. Moreover, $\dive \phi =  (-\Delta+1)^{-1}\dive \psi\in L^2(\Omega)$. Thus, $V\subseteq X$. Next, we show the density property. For this, we endow $X$ with the scalar product
 \[
    \langle u,v\rangle_X\coloneqq \langle \interior{\Curl} u,\interior{\Curl}v\rangle+\langle \dive u,\dive v\rangle + \langle u,v\rangle.
 \]
 Let $u\in V^{\bot_X}\subseteq X$. We need to show that $u=0$. For all $\psi\in C_c^\infty(\Omega)$ and $w\coloneqq (-\Delta+1)^{-1}\dive \psi$ we have
 \begin{align*}
    0 & = \langle u,\psi + \interior{\grad}w\rangle_X
    \\& = \langle \interior{\Curl} u, \interior{\Curl} \psi\rangle + \langle \dive u, \dive \psi\rangle  + \langle \dive u,\dive\interior{\grad}w\rangle + \langle u,\psi\rangle + \langle u,\interior{\grad}w\rangle
    \\& = \langle \interior{\Curl} u, \interior{\Curl} \psi\rangle + \langle \dive u, \dive \psi\rangle  + \langle \dive u,\Delta w\rangle + \langle u,\psi\rangle - \langle \dive u, w\rangle
    \\& = \langle \interior{\Curl} u, \interior{\Curl} \psi\rangle  + \langle u,\psi\rangle.
 \end{align*}
Thus, $(\interior{\Curl}^*\interior{\Curl} + 1)u=0$, which yields $u=0$.
\end{proof}
Before we come to the proof of Theorem \ref{t:gaff}, we mention an elementary formula to be used in the forthcoming proof: For all $\psi\in C_c^\infty(\Omega)^d$ we have
\[
  -\Delta I_{d\times d}  \psi = - \Dive\Grad \psi = -\Dive \Curl\psi - \grad \dive \psi.
\]
\begin{proof}[Proof of Theorem \ref{t:gaff}] By Lemma \ref{l:dens} it suffices to show the inequality for $u\in V$. For this, let $\psi\in C_c^\infty(\Omega)^d$ and put $u\coloneqq \psi + \interior{\grad}w$ with $w\coloneqq (-\Delta+1)^{-1}\dive\psi$. 
We compute
\begin{multline*}
  \|\Grad u\|^2=\|\Grad (\psi+\interior{\grad}w)\|^2\\= \langle\Grad \psi,\Grad \psi\rangle + 2\Re \langle \Grad \psi,\Grad \interior{\grad}w\rangle + \|\Grad \interior{\grad}w\|^2.
\end{multline*}
We aim to discuss every term in the latter expression separately. 
We have
\begin{align*}
  \langle\Grad \psi,\Grad \psi\rangle & = - \langle\Dive\Grad \psi, \psi\rangle
  \\  & = -\langle\Dive \Curl\psi,\psi\rangle - \langle\grad \dive \psi,\psi\rangle
  \\  & = -\langle\Dive\skew\Curl\psi,\psi\rangle + \langle \dive \psi,\dive\psi\rangle
  \\  & = \frac12\langle\Curl\psi,\Curl\psi\rangle + \langle \dive \psi,\dive\psi\rangle.
\end{align*}
Next,
\begin{align*}
 \langle \Grad \psi,\Grad \interior{\grad}w\rangle & = -\langle\Dive \Grad \psi,\interior{\grad}w\rangle
 \\ & = - \langle\Dive \Curl \psi,\interior{\grad}w\rangle-\langle\grad \dive \psi,\interior{\grad}w\rangle
 \\ & =  \langle\dive\Dive \Curl \psi,w\rangle-\langle\grad \dive \psi,\interior{\grad}w\rangle
 \\ & = -\langle\grad \dive \psi,\interior{\grad}w\rangle.
\end{align*}
By Theorem \ref{t:LaplReg}, we estimate
\[
 \|\Grad \interior{\grad}w\|^2\leq \|\Delta w\|^2=\|w-\dive\psi\|^2=\|w\|^2-2\Re\langle w,\dive\psi\rangle+\|\dive\psi\|^2.
\]
Note that since $\dive\psi\in C_c^\infty(\Omega)$, we obtain from $w=(-\Delta+1)^{-1}\dive\psi$ that
\[
  \langle \interior{\grad}w,\interior{\grad}\dive\psi\rangle + \langle w,\dive\psi\rangle = \langle \dive \psi,\dive\psi\rangle.
\]
Thus, all together, 
\begin{align*}
  \|\Grad u\|^2 & \leq \frac12\langle\Curl\psi,\Curl\psi\rangle + \langle \dive \psi,\dive\psi\rangle - 2\Re \langle\grad \dive \psi,\interior{\grad}w\rangle \\ & \quad+ \|w\|^2-2\Re\langle w,\dive\psi\rangle+\|\dive\psi\|^2
  \\ &= \frac12\langle\Curl\psi,\Curl\psi\rangle + \langle \dive \psi,\dive\psi\rangle \\ &\quad+ 2\Re\langle w,\dive\psi\rangle - 2 \langle \dive \psi,\dive\psi\rangle + \|w\|^2-2\Re\langle w,\dive\psi\rangle+\|\dive\psi\|^2
  \\ &= \frac12\langle\Curl\psi,\Curl\psi\rangle + \|w\|^2
  \\ &= \frac12\|\Curl u\|^2 + \|\dive u\|^2.\qedhere
\end{align*} 
\end{proof}

\begin{proof}[Proof of Theorem \ref{t:comp}] By Theorem \ref{t:gaff} as $B(0,1)$ is convex, we obtain that \[
\dom(A_1)\cap\dom(A_0^*)=\dom(\interior{\Curl})\cap \dom(\dive)\hookrightarrow \dom(\Grad).                                                                                    
                                                                                    \]
On the other hand $\dom(\Grad)\hookrightarrow L^2(B(0,1))^d$ is compact by Rellich's selection theorem. This yields the assertion.
\end{proof}

\begin{lemma}\label{l:toptriv} Assume the setting in \eqref{eq:setting}. Then $\kar(\dive)\cap\kar(\interior{\Curl})=\{0\}$.  
\end{lemma}
\begin{proof}
 The assertion follows from the connectedness of $B(0,1)$. See e.g.~\cite{DS52,P79}.
\end{proof}

For the next proposition, we closely follow a rationale given by Pauly and Zulehner, see \cite{PZpre}. We also refer to \cite{BPS16} for a similar argument.

\begin{proposition}\label{prop:closedminusone} Assume the setting in \eqref{eq:setting}. Then $\rge(\hat{\interior{\Curl}})\subseteq H^{-1}(\Omega)^{d\times d}$ is closed. 
\end{proposition}
\begin{proof}
 In this proof, we need to consider the differential operators on various domains. To clarify this in the notation, we attach the underlying domain as an index to the differential operators in question, that is, $\grad=\grad_\Omega$ and when the domains are considered we write $\dom(\grad)=\dom(\grad,\Omega)$ and similarly for $\rge$ and $\kar$.
 We apply Corollary \ref{cor:AB} to $A=\interior{\Curl}_{B(0,1)}$, $C=\interior{\Grad}_{B(0,1)}$. Note that $\rge(A)$ is closed by Theorem \ref{t:comp} and Theorem \ref{t:tb}. Thus, we are left with showing that
 \[
    \rge(\interior{\Curl},{B(0,1)})=\{\interior{\Curl}_{B(0,1)}\phi;\phi\in \dom(\interior{\Grad},{B(0,1)})\}.
 \]
 From Proposition \ref{prop:seq} and by Theorem \ref{t:gaff}, we infer
 \begin{align*}
     \rge(\interior{\Curl}_{B(0,1)})&=\{\interior{\Curl}_{B(0,1)}\phi;\phi\in \kar(\dive,{B(0,1)})\cap\dom(\interior{\Curl},{B(0,1)})\}
     \\ & =\{\interior{\Curl}_{B(0,1)}\phi;\phi\in \dom(\Grad,{B(0,1)})\cap\dom(\interior{\Curl},{B(0,1)})\}.
 \end{align*}
 So, let $\psi=\Curl_{B(0,1)} \phi$ for some $\phi\in \dom(\interior{\Curl},{B(0,1)})\cap \dom(\Grad,{B(0,1)})$. Extend $\phi$ and $\psi$ by zero to $B(0,2)$, we call the extensions $\phi_{\text{e}}$ and $\psi_{\text{e}}$. Note that $\phi_{\text{e}}\in \dom(\interior{\Curl},{B(0,2)})$ and $\interior{\Curl}_{B(0,2)}\phi_{\text{e}}=\psi_{\text{e}}$. By the above applied to $\Omega=B(0,2)$, we find $\phi_{\text{r}}\in \dom(\interior{\Curl},{B(0,2)})\cap\dom(\Grad,{B(0,2)})$ such that $\interior{\Curl}_{B(0,2)}\phi_{\text{r}}=\interior{\Curl}_{B(0,2)}\phi_{\text{e}}=\psi_{\text{e}}$. Thus, \[\phi_{\text{r}}-\phi_{\text{e}} \in \kar(\interior{\Curl},{B(0,2)})=\rge(\interior{\grad},{B(0,2)}),\] by Lemma \ref{l:toptriv}. Thus, we find $u\in \dom(\interior{\grad},{B(0,2)})$ with $\interior{\grad}_{B(0,2)}u=\phi_{\text{r}}-\phi_{\text{e}}$.  On $B(0,2)\setminus \overline{B(0,1)}$ we have
 \[
    0=\phi_{\text{e}}=\phi_{\text{r}}-{\grad}_{B(0,2)\setminus \overline{B(0,1)}}u.
 \]
 Therefore, ${\grad}_{B(0,2)\setminus \overline{B(0,1)}}u=\phi_{\text{r}}$ on $B(0,2)\setminus\overline{B(0,1)}$. Hence, \[u\in \dom(\Grad\grad,B(0,2)\setminus \overline{B(0,1)})=H^2(B(0,2)\setminus \overline{B(0,1)}).\] By Calderon's extension theorem, there exists \[u_{\text{e}}\in \dom(\Grad\grad,B(0,2))=H^2(B(0,2))\text{ with }u_{\text{e}}=u\text{ on }B(0,2)\setminus\overline{B(0,1)}.\] Next, we observe that $\phi_{\text{r},0}\coloneqq \phi_{\text{r}}-\grad_{B(0,2)} u_{\text{e}} \in \dom(\Grad,B(0,2))$ as well as $u-u_{\text{e}}\in \dom(\grad,B(0,2))$ and
  \[
    \phi_{\text{r}}= \phi_{\text{r},0} - \grad_{B(0,2)}(u-u_{\text{e}}).
 \]
 Moreover, on $B(0,2)\setminus \overline{B(0,1)}$, we have $\phi_{\text{r},0}=0$ as well as $u-u_{\text{e}}=0$. Thus, $\phi_{\text{r},0}\in \dom(\interior{\Grad},B(0,1))$ and $u-u_{\text{e}}\in \dom(\interior{\grad},B(0,1))$. Thus, 
 \[
    \psi = \Curl_{B(0,1)} \phi =\Curl_{B(0,1)} \phi_\textnormal{r} = \Curl_{B(0,1)}(\phi_{\text{r},0} - \interior{\grad}_{B(0,1)}(u-u_{\text{e}}))=\interior{\Curl}_{B(0,1)} \phi_{\text{r},0}.
 \]
 Therefore,
 \begin{align*}
  \rge(\interior{\Curl},{B(0,1)}) & =\{\interior{\Curl}_{B(0,1)}\phi;\phi\in \dom(\interior{\Grad},{B(0,1)})\cap\dom(\interior{\Curl},B(0,1))\} 
  \\&=\{\interior{\Curl}_{B(0,1)}\phi;\phi\in \dom(\interior{\Grad},{B(0,1)})\}.\qedhere
 \end{align*}
\end{proof}

\begin{lemma}\label{l:distr} Let $\Omega\subseteq \mathbb{R}^d$ open, bounded, $\phi\in L^2(\Omega)^d$ with $\spt\phi\subseteq \Omega$. Then
\[
   \dom(\interior{\Dive}\skew)^* \ni \Curl \phi = \interior{\Curl}\phi\in \dom(\Dive\skew)^*
\]
\end{lemma}
\begin{proof}
 We have $ \dom({\Dive}\skew)^* \hookrightarrow \dom((\interior{\Dive}\skew)^*$. Let $\eta\in C_c^\infty(\Omega)$ with the property $\eta=1$ on $\spt \phi$. Then for all $\psi\in\dom(\Dive\skew)$ we have $\eta\psi\in\dom(\interior{\Dive}\skew)$ and so, 
 \begin{align*}
     \langle\interior{\Curl}\phi,\psi\rangle & =\langle\phi,2\Dive\skew\psi\rangle 
     \\&=\langle\phi,2\Dive\skew\eta\psi\rangle
     \\&=\langle\phi,2\interior{\Dive}\skew\eta\psi\rangle
     \\&=\langle\Curl\phi,\eta\psi\rangle.
 \end{align*}
 Thus, there is $\kappa>0$ such that for all $\psi\in \dom(\Dive \skew)$
 \begin{align*}
     |(\interior{\Curl}\phi)(\psi)|&=|(\Curl(\phi)(\psi))|
     \\ & =|(\Curl(\phi)(\eta\psi))|
     \\ & \leq \kappa \|\psi\|_{\dom({\Dive}\skew)}.
 \end{align*} This yields the assertion.
\end{proof}

Finally, we can prove the $\dive$-$\curl$ lemma with operator-theoretic methods. We shall also formulate a simpler version of the $\dive$-$\curl$ lemma, which needs less technical preparations. In fact, the simpler version only uses Theorem \ref{t:mrcc} and Theorem \ref{t:comp}.

\begin{proof}[Proof of Theorem \ref{t:dcl2}] We apply Theorem \ref{t:mr} with the setting in \eqref{eq:setting}. For this, by Lemma \ref{l:distr}, we note that $\Curl v_n = \interior{\Curl}v_n=\hat{\interior{\Curl}}\,v_n$. With Theorem \ref{t:mr} at hand, we need to establish that $(\hat{\interior{\Curl}}\,v_n)_n$ is relatively compact in $\dom(\interior{\Curl}^*)^*$. By Corollary \ref{cor:AB} applied to $C=A=\interior{\Curl}^*$, the latter is the same as showing that $(\hat{\interior{\Curl}}\,v_n)_n$ is relatively compact in $\rge(\hat{\interior{\Curl}})$. On the other hand, by Proposition \ref{prop:closedminusone}, $\rge(\hat{\interior{\Curl}})$ is closed in $H^{-1}(\Omega)^{d\times d}$. Thus, since $(\hat{\interior{\Curl}}\,v_n)_n$ is relatively compact in $H^{-1}(\Omega)^{d\times d}$, we get that $(\hat{\interior{\Curl}}v_n)_n$ is relatively compact in $\dom(\interior{\Curl}^*)^*$. This yields the assertion.
\end{proof}

Theorem \ref{t:mrcc} with the setting in \eqref{eq:setting} reads as follows. Note that the assertion follows from Theorem \ref{t:comp}.

\begin{theorem}\label{t:dclcc} Let $(u_n)_n$ in $\dom(\dive)$ and $(v_n)_n$ in $\dom(\interior{\Curl})$ be weakly convergent sequences. Then
\[
  \lim_{n\to\infty}\langle u_n,v_n\rangle_{L^2(\Omega)^d} = \langle \lim_{n\to\infty} u_n,\lim_{n\to\infty}v_n\rangle_{L^2(\Omega)^d}.
\] 
\end{theorem}

It is well-known that the sequence property and the compactness of the sequence is true also for submanifolds of $\mathbb{R}^d$ and the covariant derivative on tensor fields of appropriate dimension and its adjoint. We conclude this exposition with a less known sequence. The Pauly--Zulehner $\Grad\grad$-complex, see \cite{PZ17}.
\section*{An Example -- the Pauly--Zulehner-$\Grad\grad$-complex}

In the whole section, we let $\Omega\subseteq \mathbb{R}^3$ to be a bounded Lipschitz domain. We will denote by $\curl$ the usual $3$-dimensional curl operator that maps vector fields to vector fields. Some definitions are in order
\begin{definition}We define
\begin{align*}
     \intersec{\grad_\textnormal{r}\grad} &\colon \inter{H}^2(\Omega) \subseteq L^2(\Omega)\to L_{\sym}^2(\Omega),\phi\mapsto \grad_\textnormal{r}\grad\phi.
     \\ \interior{\curl}_{\textnormal{r},\sym} &\colon \dom(\interior{\curl}_{\textnormal{r}})\cap L_{\text{sym}}^2(\Omega)\subseteq L_{\sym}^2(\Omega)\to L_{\dev}^2(\Omega) ,\phi\mapsto \interior{\curl}_{\textnormal{r}}\phi
     \\ 
     \interior{\dive}_{\textnormal{r},\dev} &\colon \dom(\interior{\dive}_{\textnormal{r}})\cap L_{\dev}^2(\Omega)\subseteq L_{\sym}^2(\Omega)\to L^2(\Omega)^d ,\phi\mapsto \Dive\phi
      \\ \ob{\dive\dive_{\textnormal{r}}}_{,\sym}& \colon \dom(\ob{\dive\Dive}_{\sym}) \subseteq L_{\sym}^2(\Omega) \to L^2(\Omega), \phi\mapsto \dive\Dive \phi,
      \\ \sym\curl_{\textnormal{r},\dev} & \colon\dom(\curl_\textnormal{r})\cap L_{\dev}^2(\Omega)\subseteq L_{\text{dev}}^2(\Omega)\to L^2_{\text{sym}}(\Omega), \phi\mapsto \sym\curl_\textnormal{r}\phi,
     \\ \dev\grad_\textnormal{r} & \colon H^1(\Omega)^3 \subseteq L^2(\Omega)^3\to L^2_{\dev}(\Omega),\phi\mapsto \dev\grad_\textnormal{r}\phi.
\end{align*}
The subscript $\textnormal{r}$ refers to row-wise application of the vector-analytic operators, where it is attached. Moreover, as before, we have attached a ``$\interior{\ }$'' above the differential operators in question, if we consider the completion of smooth tensor fields with compact support with appropriate norm. The operators $\dev$ and $\sym$ are the projections on the \emph{deviatoric} and \emph{symmetric} parts of $3\times3$-matrices, that is, for a matrix $A\in \mathbb{C}^{3\times 3}$, we put
\[
   \dev A\coloneqq A-\frac{1}{3}\tr (A)I_{3\times 3},\quad \sym A= \frac{1}{2}(A+A^T).
\]
Moreover, we define $L^2_{\dev}(\Omega)\coloneqq \dev\left[L^2(\Omega)^{3\times 3}\right]$ as well as $L^2_{\sym}(\Omega)\coloneqq \sym\left[L^2(\Omega)^{3\times 3}\right]$.
\end{definition}

Next, we gather some of the main results of Pauly--Zulehner:

\begin{theorem}[{{\cite[Lemma 3.5, Remark 3.8, and Lemma 3.21]{PZ17}}}] The pairs
\begin{multline*}
   \left(\intersec{\grad_\textnormal{r}\grad},\interior{\curl}_{\textnormal{r},\sym}\right),\; \left(\interior{\curl}_{\textnormal{r},\sym},\interior{\dive}_{\textnormal{r},\dev}\right),\\
   \left(-\dev\grad_\textnormal{r},\sym\curl_{\textnormal{r},\dev}\right),\;\left(\sym\curl_{\textnormal{r},\dev},\ob{\dive\dive_\textnormal{r}}_{,\sym}\right)
\end{multline*}
are compact sequences. Moreover, we have $\intersec{\grad_\textnormal{r}\grad}^*=\ob{\dive\dive_\textnormal{r}}_{,\sym}$, $\interior{\curl}_{\textnormal{r},\sym}^*=\sym\curl_{\textnormal{r},\dev}$, $\interior{\dive}_{\textnormal{r},\dev}^*=-\dev\grad_\textnormal{r}$. 
\end{theorem}

We have now several theorems being consequences of our general observation in Theorem \ref{t:mr}. We will formulate the versions for Theorem \ref{t:mr} only. The analogues to Theorem \ref{t:mrcc} are straightforwardly written down, which we will omit here.
\begin{theorem} \begin{enumerate}[label=(\alph*)]
                 \item Let $(u_n)_n,(v_n)_n$ be weakly convergent sequences in $L^2_{\sym}(\Omega)$. Assume that 
                 \[
                     (\ob{\dive\dive_\textnormal{r}}_{,\sym}u_n)_n,  (\interior{\curl}_{\textnormal{r},\sym}v_n)_n
                 \]
                 are relatively compact in $\dom(\intersec{\grad_\textnormal{r}\grad})^*$ and $\dom(\sym\curl_{\textnormal{r}})^*$. Then
                 \[
                   \lim_{n\to\infty}\langle u_n,v_n\rangle = \langle \lim_{n\to\infty} u_n,\lim_{n\to\infty} v_n\rangle. 
                 \]
                 \item Let $(u_n)_n,(v_n)_n$ be weakly convergent sequences in $L^2_{\dev}(\Omega)$. Assume that 
                 \[
                     (\sym\curl_{\textnormal{r},\dev}u_n)_n,  (\interior{\dive}_{\textnormal{r},\dev}v_n)_n
                 \]
                 are relatively compact in $\dom(\interior{\curl}_{\textnormal{r},\sym})^*$ and $\dom(\dev\grad_\textnormal{r})^*$. Then
                 \[
                   \lim_{n\to\infty}\langle u_n,v_n\rangle = \langle \lim_{n\to\infty} u_n,\lim_{n\to\infty} v_n\rangle. 
                 \]
                \end{enumerate}
\end{theorem}

\section*{Acknowledgements}

This work was carried out with financial support of the EPSRC grant EP/L018802/2: “Mathematical foundations of metamaterials: homogenisation, dissipation and operator theory”. A great deal of this research has been obtained during a research visit of the author at the RICAM for the special semester 2016 on Computational Methods in Science and Engineering organised by Ulrich Langer and Dirk Pauly, et al. The wonderful atmosphere and the hospitality extended to the author are gratefully acknowledged.

\noindent
Marcus Waurick \\Department of Mathematical Sciences, University of Bath,\\
Claverton Down, Bath, BA2 7AY,\\
United Kingdom\\
Email: %M.Waurick@bath.ac.uk
{\tt m.wau\rlap{\textcolor{white}{hugo@egon}}rick@bath.\rlap{\textcolor{white}{darmstadt}}ac.uk}

\end{document}